\documentclass[11pt]{article}

\usepackage{amssymb,amsmath,amsthm}
\usepackage[left=3.2cm,top=3.5cm,right=3.2cm,bottom=3cm]{geometry}
\usepackage{bbm}
\usepackage{graphicx,color,pgf,transparent}

\def\XXint#1#2#3{{\setbox0=\hbox{$#1{#2#3}{\int}$}
\vcenter{\hbox{$#2#3$}}\kern-.5\wd0}}

\newtheorem{theorem}{Theorem}[section]
\newtheorem{lemma}[theorem]{Lemma}
\newtheorem{proposition}[theorem]{Proposition}
\newtheorem{corollary}[theorem]{Corollary}

\theoremstyle{definition}

\newtheorem{remark}[theorem]{Remark}

\numberwithin{equation}{section}

\newcommand{\R}{\mathbb{R}}
\renewcommand{\S}{\mathbb{S}}
\newcommand{\C}{\mathbb{C}}

\newcommand{\eps}{\varepsilon}

\newcommand{\vphi}{\varphi}

\DeclareMathOperator{\Bal}{Bal}

\DeclareMathOperator{\diverg}{div}
\renewcommand{\div}{\diverg}
\DeclareMathOperator{\grad}{grad}
\DeclareMathOperator{\supp}{supp}

\DeclareMathOperator{\IM}{Im}
\DeclareMathOperator{\m}{m}
\renewcommand{\Im}{\IM}

\title{An equilibrium problem on the sphere with two equal charges}
\author{Juan G. Criado del Rey and Arno B.J. Kuijlaars}
\date{\today}

\begin{document}

\maketitle

\begin{abstract}
We study the equilibrium measure on the two dimensional
sphere in the presence of  an external field generated by 
two equal point charges. The support of the equilibrium
measure is known as the droplet. Brauchart et al.\ showed
that the complement of the droplet consists of two
spherical caps when the charges are small.
When the charges are bigger the droplet becomes
simply connected and we prove that the boundary of
the droplet is mapped by stereographic projection to 
an ellipse in the plane. 

Moreover, we compute a mother body for the droplet that 
we derive from an equilibrium problem with a weakly 
admissible  external field on the real line.
\end{abstract}

\section{Introduction and statement of results}

\subsection{Main result}
Let us denote by $\S^2 = \{x\in\R^3 : \|x\| = 1\}$ the two dimensional
unit sphere, where $\|\cdot\|$ stands for the usual Euclidean norm.
We use $\lambda$ for the Lebesgue measure on  $\S^2$, normalized so that
$\lambda(\S^2) = 1$, and for a closed subset $D \subset \S^2$ we 
write $\lambda_D$ for the Lebesgue measure restricted to $D$.

An \emph{external field} is a function $Q:\S^2\to \R\cup\{+\infty\}$ that is lower semicontinuous and finite at least on a set of 
positive logarithmic capacity. The \emph{weighted logarithmic energy}
in the presence of $Q$ of a probability measure $\mu$ on $\S^2$ 
is given by
\begin{equation}\label{eq:def_energy}
I_Q[\mu] = \iint_{\S^2 \times \S^2}\log\frac{1}{\|x-y\|}d\mu(x)d\mu(y)+2\int_{\S^2} Q(x)d\mu(x).
\end{equation} 
Then $I_Q$ represents the total energy of a large system of 
mutually repelling particles with the external field $Q$ acting 
on them. It is a well known result  (see \cite[Theorem 1.2]{DragnevSaff2007}, but the proof is the same as in
the case of a weighted logarithmic energy problem in the complex plane \cite{SaffTotik1997}) that there is a unique probability measure $\mu_Q$, called the \emph{equilibrium measure}, that minimizes  \eqref{eq:def_energy} among all probability measures on $\S^2$.

The equilibrium measure is characterized by  Frostman-type \cite{Frostman1935} 
variational conditions which say that for some constant $\ell_Q$,
\begin{equation} \label{eq:frostman}
	\begin{cases}
	Q + U^{\mu_Q} \leq \ell_Q & \text{ on } \supp(\mu_Q), \\
	Q + U^{\mu_Q} \geq \ell_Q & \text{ quasi-everywhere on } \S^2,
	\end{cases}
	\end{equation}
where quasi-everywhere means up to a set of zero logarithmic capacity,
and $U^{\mu_Q}$ is the logarithmic potential, which for general
positive measures $\sigma$ is defined by 
\begin{equation}\label{eq:def_potential}
	U^\sigma(x) = \int_{\S^2}\log\frac{1}{\|x-y\|}d\sigma(y).
\end{equation} 
If there is no exceptional set of zero logarithmic capacity then \eqref{eq:frostman}
takes the form
\begin{equation} \label{eq:frostmaneasy}
\begin{cases}
Q + U^{\mu_Q} = \ell_Q & \text{ on } \supp(\mu_Q), \\
Q + U^{\mu_Q} \geq \ell_Q & \text{ on } \S^2,
\end{cases}
\end{equation}
and this will be the situation for the external fields considered in this paper.

The precise determination of the equilibrium measure or its support is usually a very hard problem for a general (and even for a not too trivial) choice of $Q$. The problem becomes somewhat more
tractable if we consider external fields that
are logarithmic potentials \eqref{eq:def_potential} of
highly concentrated measures $\sigma$.
We shall refer to $\sigma$ as a charge distribution and we
use $\mu_{\sigma}$ to denote the equilibrium measure in the
external field $U^{\sigma}$. We are going to determine
$\mu_{\sigma}$ explicitly for the external field
\begin{equation} \label{eq:def_external_field} 
	Q(x) =  U^{\sigma}(x), \qquad \sigma = a \delta_{p_1} + a \delta_{p_2}
\end{equation}
with $p_1, p_2 \in \S^2$ and $a > 0$.

 For discrete charge distributions $\sigma$ like the one in \eqref{eq:def_external_field} it is known that  
\begin{equation} \label{eq:disjoint_supports}
	 \supp(\mu_{\sigma}) \cap \supp(\sigma) = \emptyset 
\end{equation} 
and whenever \eqref{eq:disjoint_supports} holds the equilibrium measure
takes the form 
\begin{equation} \label{eq:musigma} 
	\mu_\sigma = \lambda(D)^{-1}  \lambda_D, \end{equation}
 where $D = \supp(\mu_{\sigma})$ is some compact set called the \emph{droplet}. 
In other words, $\mu_\sigma$ has constant density with respect
to Lebesgue measure on its support and therefore the droplet $D$ 
is all one needs to know in order to solve the equilibrium problem. 
See \cite[Theorem 11]{Gustafsson2018}\footnote{It is Theorem 6.3 in the preprint arXiv:1605.03102.} 
for a general version of this statement in terms of partial balayage
on compact manifolds. Moreover, it follows from the general theory 
that 
\begin{equation}\label{eq:volume_droplet}
	\lambda(D) = \frac{1}{1+ \m(\sigma)},
\end{equation} 
where $\m(\sigma) = \int d\sigma$ is the total mass of $\sigma$.
For ease of reference we give the (simple) proof 
of \eqref{eq:musigma} and \eqref{eq:volume_droplet} in 
the appendix \ref{subsec:musigma_lemma}.

The case $\sigma = a\delta_p$, where $a>0$ and $p\in \S^2$, was 
studied in \cite{Dragnev2002}. For this $\sigma$ the droplet is the
complement of a spherical cap $B$ centered at $p$ with area
\[ 	\lambda(B) = \frac{a}{1+a}. \] 
A clever choice of the parameter $a$ leads in that article to a bound 
on the separation of minimal logarithmic energy points. A similar
technique is employed in \cite{DragnevSaff2007} and \cite{Brauchart2018}
to prove bounds on the separation of minimal Riesz $s$--energy points 
on the $d$--dimensional sphere $\S^d$, and in \cite{CriadoDelRey2019} 
to bound the separation of minimal Green energy points on 
compact manifolds. 

In \cite{Brauchart2018} the authors consider the case where 
$\sigma = \sum\limits_{j=1}^m a_j\delta_{p_j}$, with $a_j > 0$ 
and $p_j \in \S^2$ for every $j$, is a combination of point masses 
(not only for the logarithmic case, but also for the Riesz 
interaction). They prove that if the $a_j$ are sufficiently small, 
then
\begin{equation} \label{eq:spherical_caps} 
	D = \S^2 \setminus \left( \bigcup_{j=1}^m B_j \right) 
	\end{equation}
where $B_j$ is the (open) spherical cap centered at the point $p_j$ 
with area
\[ \lambda(B_j) = \frac{a_j}{1+a_1+\cdots + a_m}.
\] 
The result \eqref{eq:spherical_caps} holds if and only if the spherical caps are mutually disjoint. 

The equilibrium problem becomes more complicated (and maybe more
interesting) when some of the 
$a_j$ grow large enough so that two (or more) of the spherical caps
start to overlap. In such a case the complement of the droplet 
is no longer a union of spherical caps, but some larger set as
in Figure \ref{fig:fig1}.

\begin{figure}[t]
\centering
\includegraphics[width=0.52\textwidth]{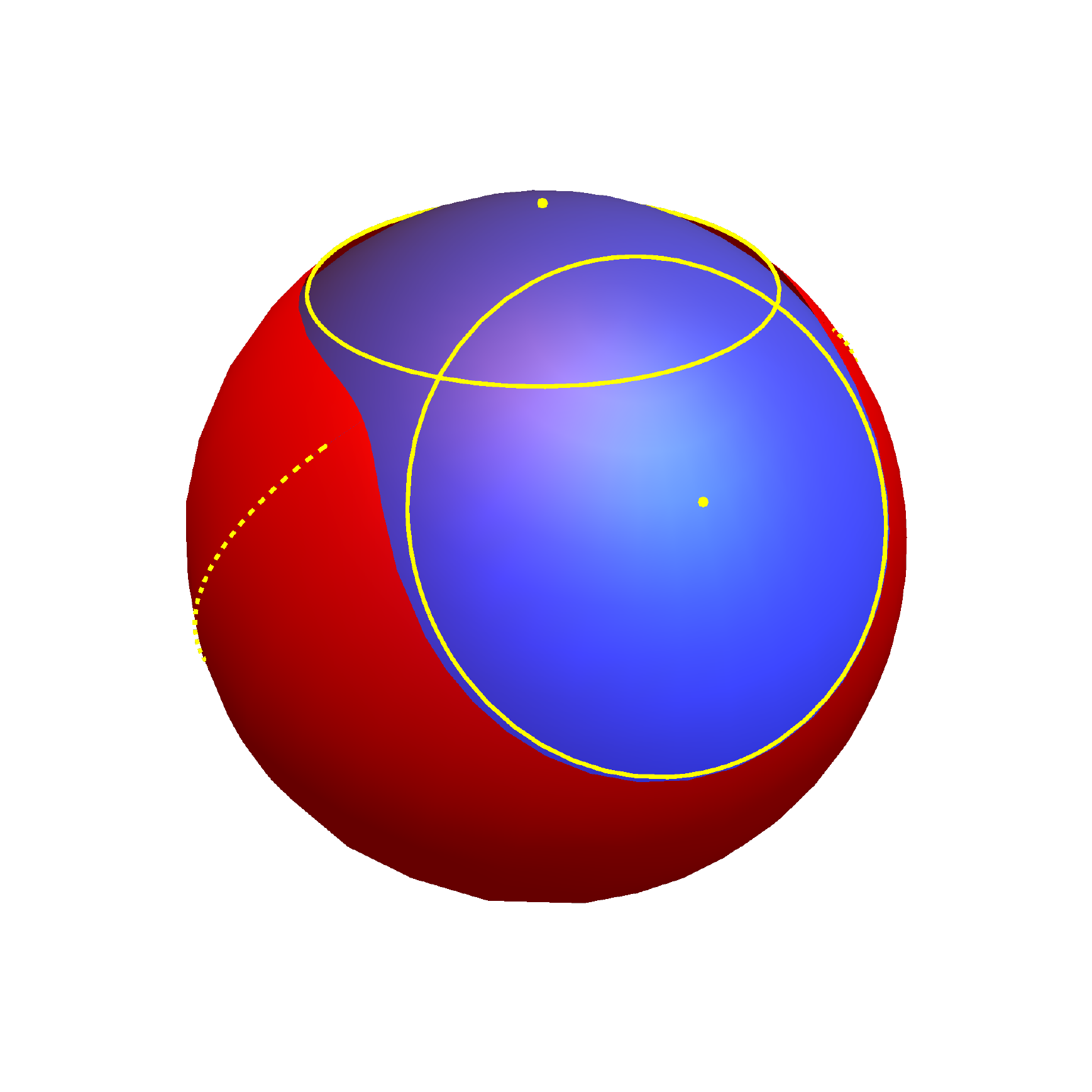}\hspace{-1cm}
\includegraphics[width=0.52\textwidth]{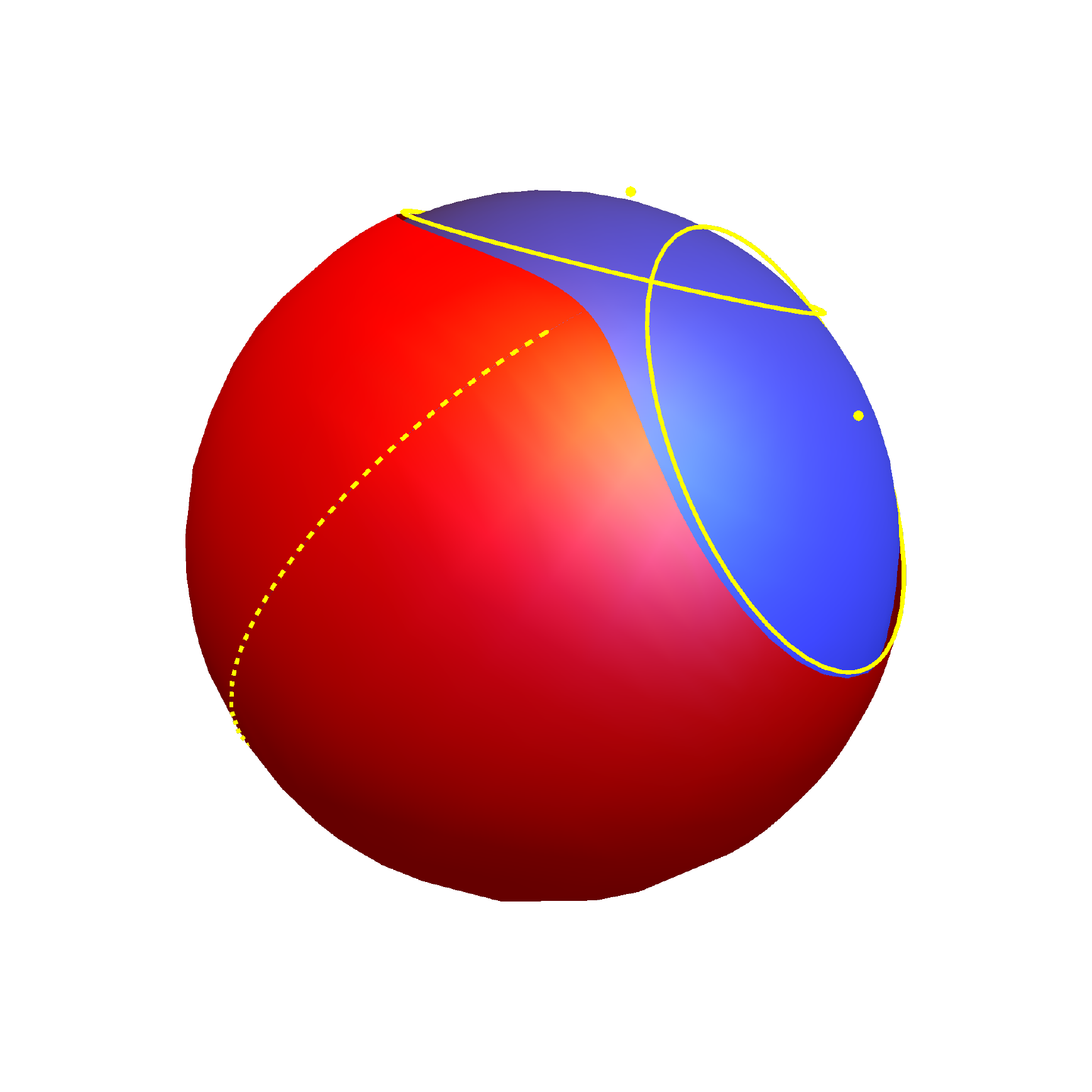}
\caption{Picture of the droplet (red region) obtained using the formula \eqref{eq:equation_ellipse} from two different viewpoints. The 
	spherical caps centered at $p_1$ and $p_2$ with geodesic radii $a/(1+2a)$ are also represented, as well as the support for the mother body (dashed line inside the droplet). Compare this picture with Figure 4 in \cite{Brauchart2018}.}\label{fig:fig1}
\end{figure}

The goal of this paper is to study the symmetric instance of two
point charges, i.e., $\sigma = a_1\delta_{p_1}+a_2\delta_{p_2}$
with $a_1 = a_2 = a$ but the charge $a$ can be arbitrarily large. 
In other words, for us
\begin{equation} \label{eq:Q_two_charges}
Q(x) = U^{\sigma}(x) = a\log\frac{1}{\|x-p_1\|}+a\log\frac{1}{\|x-p_2\|}, \quad a > 0,
\end{equation}
where $p_1$ and $p_2$ are two distinct points on $\S^2$. 

We are going to use the stereographic projection $\phi:\S^2 \to\C \cup \{\infty\}$ given by
\[
\phi(x_1,x_2,x_3) =  \begin{cases} \frac{x_1+ix_2}{1-x_3},
	& \text{for } (x_1, x_2, x_3) \in  \S^2, \, x_3 \neq 1, \\
	\infty, & \text{for } (x_1,x_2,x_3) = (0,0,1). 
	\end{cases}
\] 
We assume, without loss of generality, that $p_1$ and $p_2$ are mapped to two purely imaginary complex numbers 
\begin{equation} \label{eq:ib} 
	\phi(b_1) = ib, \qquad \phi(p_2) = -ib, \qquad \text{with } b \geq 1. 
\end{equation}
We achieve \eqref{eq:ib} by rotating the sphere such
that  $p_1 = (0, x_2, x_3)$ and $p_2 = (0,-x_2,x_3)$ are 
symmetric with respect to the north pole and $x_3 \geq 0$. 

We consider $b$ fixed and study the effect on the droplet when $a$ varies.
In this setting the critical value is
\begin{equation} \label{eq:acr}
	a_{cr} = (b^2-1)^{-1} \end{equation}
in the sense that $a \leq a_{cr}$ corresponds to the situation when the charge $a$ is small enough so that the droplet is the complement of two spherical caps (see \cite[Theorem 1]{Brauchart2018}).
For $a > a_{cr}$ we find the surprising result that
the stereographic projection of the droplet is an \emph{ellipse},
and this is our main result.

\begin{theorem}\label{thm:main} 
Let $Q$ be given by \eqref{eq:Q_two_charges} where $p_1$
and $p_2$ are two points on the unit sphere, that are mapped by
stereographic projection to $ib$ and $-ib$, respectively, with 
$b \geq 1$. Suppose $a > a_{cr} = (b^2-1)^{-1}$. 
Let $D$ be the support of the equilibrium measure
and $\phi(D) = \Omega$. Then $\Omega$ is the compact region 
in the complex plane enclosed by the ellipse with equation
\begin{equation}\label{eq:equation_ellipse}
\frac{2(b^2a-a-1)}{b^2+1}x^2+\frac{2(b^2a+a+1)}{b^2-1}y^2 = 1.
\end{equation}
\end{theorem}

We are going to prove Theorem \ref{thm:main} by explicitly verifying
the variational conditions \eqref{eq:frostmaneasy}. Namely,
if $\Omega$ is the region enclosed by the ellipse \eqref{eq:equation_ellipse}, and $D = \phi^{-1}(\Omega)$ then
we prove for some constant $\ell_a$,
\begin{equation} \label{eq:frostman_lambdaD}
(1+2a) U^{\lambda_D}(x) + a \log \frac{1}{\|x - p_1\|}
+ a \log \frac{1}{\|x-p_2\|}
\begin{cases} = \ell_a & \text{ on } D, \\
\geq \ell_a & \text{ on } \S^2.
\end{cases}
\end{equation}

For $a = a_{cr}$ the equation \eqref{eq:equation_ellipse} reduces to 
$\frac{4 b^2}{(b^2-1)^2} y^2 = 1$ and
the domain $\Omega$ tends to the horizontal strip 
\begin{equation} \label{eq:Omegacr} 
\Omega_{cr} = \left\{ z \in \C : |\Im z| \leq \frac{b^2-1}{2b} \right\}. 
\end{equation}
as $a \to a_{cr}^+$. Then $D_{cr} = \phi^{-1}(\Omega_{cr})$
is the complement of two spherical caps centered at $p_1$
and $p_2$ that are tangent at the north pole.
The conditions \eqref{eq:frostman_lambdaD} hold in
this critical case, and this follows from the results in \cite{Brauchart2018}.

\subsection{A dual weighted energy problem}

We are able to compute the droplet $D$ because of duality between
$(\sigma, D)$ and $(\sigma^*,D^*)$ where
\[ D^* = \overline{\S^2 \setminus D} \]
and $\sigma^*$ is a measure on $D$ that will be such that
\begin{equation} \label{eq:musigmastar} 
	\mu_{\sigma^*} = \lambda(D^*)^{-1} \lambda_{D^*}, \end{equation}
which is the relation dual to \eqref{eq:musigma}. In addition,
the  measure $\sigma^*$ will be highly concentrated, although not a
finite combination of point masses. It will be supported
on $\phi^{-1}(\R \cup \{\infty\})$ which is the great circle
on $\S^2$ containing those points of $\S^2$ that have equal
distance to $p_1$ and $p_2$, see Figure \ref{fig:fig1}, where the support of $\sigma^*$ is represented by a dashed line inside the droplet.
Because of \eqref{eq:volume_droplet}
and $\lambda(D^*) = 1 - \lambda(D)$, we will have
\begin{equation} \label{eq:masses} 
\begin{aligned}
	\m(\sigma) &= 2a, & \lambda(D) & = \frac{1}{1+2a}, \\
	\m(\sigma^*) &= \frac{1}{2a}, \qquad  & \lambda(D^*) & = \frac{2a}{1+2a}.
	\end{aligned}
\end{equation}
The measure $\sigma^*$ is related to a \emph{mother body}
(sometimes called potential theoretic skeleton \cite{Gustafsson2014}) for the domain $D$ 
in the sense that the relations \eqref{eq:equality_potentials}--\eqref{eq:inequality_potentials}
from the next theorem hold. The identity \eqref{eq:equality_potentials} 
specifies that, in the complement of $D$, the two probability measures
$2a\sigma^*$ and $(1+2a) \lambda_D$ have the same logarithmic
potentials, up to a constant.

\begin{theorem}\label{thm:quadrature}
Let $a>0$ be arbitrary. There is a positive measure $\sigma^*$ 
supported on $\phi^{-1}(\R \cup \infty)$ with total mass $\m(\sigma^*) = \frac{1}{2a}$ such that
for some constant $m$,
\begin{align} 
U^{\lambda_D}(x) = U^{\frac{2a}{1+2a}\sigma^*}(x)+m \qquad & \text{ if }x\in \S^2\setminus D,\label{eq:equality_potentials}\\
U^{\lambda_D}(x) \leq U^{\frac{2a}{1+2a}\sigma^*}(x)+m \qquad & \text{ if }x\in D.\label{eq:inequality_potentials}
\end{align}

The measure $\sigma^*$ is given by  
\begin{equation} \label{eq:sigmastar} 
	\sigma^* = \frac{1}{2a} \left( \phi^{-1}\right)_* \left(\mu_V \right),
\end{equation}
where $\mu_V$ is a probability measure on $\R$ that
will be described in Theorem \ref{thm:equilibrium_muV} below.
Thus $2a \sigma^*$ is the pushforward by the inverse
stereographic projection $\phi^{-1}$ of $\mu_V$.
\end{theorem}

The properties \eqref{eq:equality_potentials}-\eqref{eq:inequality_potentials}
of $\sigma^*$ express that the complement of the droplet solves the weighted energy problem for the measure $\sigma^*$.
\begin{corollary} \label{cor:mother_body}
	We have $\mu_{\sigma^*} = \lambda(D^*)^{-1} \lambda_{D^*}$.
	That is, if we set $Q^*(x) = U^{\sigma^*}(x)$ and $D^* = \overline{\S^2\setminus D}$, then the probability measure $\mu_{\sigma^*} = \lambda(D^*)^{-1}\lambda_{D^*}$ minimizes the weighted energy functional $I_{Q^*}[\mu]$.
\end{corollary}
\begin{proof}
From \eqref{eq:equality_potentials}, \eqref{eq:masses},
and the fact that $U^{\lambda} = \ell_0$ is constant on $\S^2$, we find
\begin{align}  \label{eq:varconditions_sigmastar}
	U^{\sigma^*} + U^{\lambda(D^*)^{-1} \lambda_{D^*}}
	& = \frac{1+2a}{2a} \left( U^{\lambda_D}  - m \right)  
		+ \frac{1+2a}{2a} U^{\lambda_{D^*}}  \qquad \text{ on } D^* \\
	& = \frac{1+2a}{2a} \left( U^{\lambda} -  m \right) 
	= \frac{1+2a}{2a} (\ell_0 - m ),  \nonumber
\end{align}
while  \eqref{eq:inequality_potentials} implies that 
inequality $\geq$  in \eqref{eq:varconditions_sigmastar}
holds on $D$. Thus $\lambda(D^*)^{-1} \lambda_{D^*}$ satisfies the variational conditions 
\eqref{eq:frostmaneasy}  associated with the weighted energy
problem with external field $U^{\sigma^*}$ and the corollary follows.
\end{proof}

\begin{remark} 
	For $\sigma$ and $\lambda_{D^*}$ we have the similar relations,
	for some constant $m^*$,
	\begin{align} 
	U^{\lambda_{D^*}}(x) = U^{\frac{1}{1+2a}\sigma}(x)+m^* \qquad & \text{ if }x\in \S^2\setminus D^*,\label{eq:equality_potentials2}\\
	U^{\lambda_{D^*}}(x) \leq U^{\frac{1}{1+2a}\sigma}(x)+m^* \qquad & \text{ if }x\in D^*, \label{eq:inequality_potentials2}
	\end{align}  
which follows from \eqref{eq:frostman_lambdaD} with similar arguments as in
the proof of Corollary \ref{cor:mother_body}.
Thus \[
\frac{1}{1+2a}\sigma = \frac{a}{1+2a}\delta_{p_1}+\frac{a}{1+2a}\delta_{p_2}
\] is a mother body for the complement of the droplet.
\end{remark}

\begin{remark}
The complement of the droplet can be shown to be a quadrature domain 
for subharmonic functions. This is, for every function $\vphi$ continuous in $\overline{\S^2\setminus D}$ and subharmonic in $\S^2\setminus D$, we have
\begin{equation}\label{eq:subharmonic_quadrature}
\int_{\S^2\setminus D} \vphi\,d\lambda \geq \frac{a}{1+2a}\vphi(p_1)+\frac{a}{1+2a}\vphi(p_2).
\end{equation} 
In particular, if $\vphi$ is harmonic in $\S^2\setminus D$, then \eqref{eq:subharmonic_quadrature} becomes an equality.
The inequality \eqref{eq:subharmonic_quadrature} for subharmonic
functions is essentially a consequence of \eqref{eq:equality_potentials2}--\eqref{eq:inequality_potentials2}.

Likewise, it is a consequence of \eqref{eq:equality_potentials}--\eqref{eq:inequality_potentials}
that the droplet itself is also a quadrature domain in a generalized sense,
namely
\[
\int_D \vphi\,d\lambda \geq \frac{2a}{1+2a}\int_{\phi^{-1}(\R)} \vphi\,d\sigma^*
\] for every $\vphi$ continuous in $D$ and subharmonic in the interior 
of $D$. 
\end{remark}

\subsection{A weakly admissible external field on $\R$}

We obtain $\sigma^*$ from a weakly admissible external
field on the real line. For measures $\mu$ on the real
line we continue to use the notation as in \eqref{eq:def_energy} 
and \eqref{eq:def_potential} for their weighted logarithmic energy
and logarithmic potential,
where now of course the integration is over the real line
and $\|x-y\|$ is simply the absolute value $|x-y|$. 

The equilibrium measure $\mu_V$ in the external
field  $V : \R \to \R \cup \{\infty\}$ is the unique 
minimizer of the energy functional
\[ I_V[\mu]  = 
\iint_{\R \times \R} \log \frac{1}{|x-y|} d\mu(x) d\mu(y)
+ 2 \int_{\R} V(x) d\mu(x) \]
among probability measures on $\R$.
We continue to use $a_{cr}= (b^2-1)^{-1}$ as in \eqref{eq:acr}.

\begin{theorem} \label{thm:equilibrium_muV} 
For $a > 0$ and $b \geq 1$, the equilibrium measure in the external field
\begin{equation} \label{eq:def_V}
	V(x) = \frac{1+a}{2} \log\left(x^2 + b^{-2}\right)
		- \frac{a}{2} \log\left(x^2 + b^2\right)  
	\end{equation}
on $\R$ is the probability measure $\mu_V$ on $\R$, 
which is described as follows:
\begin{itemize}
	\item[\rm (i)] If $a\leq a_{cr}$, then $\mu_V$ is the probability measure supported on the full real line with density
	\begin{equation}\label{eq:density_precrit}
	\frac{d\mu_V}{dx} = \frac{(1+a-b^2a)x^2+b^2+b^2a-a}{\pi b(x^2+b^2)(x^2+b^{-2})}, \qquad x \in \R.
	\end{equation}
	
	\item[\rm (ii)] If $a> a_{cr}$, then $\mu_V$ is the probability measure supported on $[-A,A]\subset \R$ with density
	\begin{equation}\label{eq:density_postcrit}
	\frac{d\mu_V}{dx} = \frac{\sqrt{C}\sqrt{A^2-x^2}}{\pi(x^2+b^2)(x^2+b^{-2})},
	\qquad x \in [-A,A], 
	\end{equation}
	where
	\begin{equation}
	\begin{aligned}
	\label{eq:AC_postcrit}
	A & = \frac{b\sqrt{1+2a}}{\sqrt{b^2a+a+1}\sqrt{b^2a-a-1}}, \\
	C & = \frac{(b^4-1)(b^2a+a+1)(b^2a-a-1)}{b^4}.
	\end{aligned} 
	\end{equation}
\end{itemize}
\end{theorem}
In case $a > a_{cr}$ it can be checked that $\pm A$ 
where $A$ is given by \eqref{eq:AC_postcrit} are
the foci of the ellipse \eqref{eq:equation_ellipse}.
If $a = a_{cr}$, then \eqref{eq:density_precrit} reduces to
\begin{equation}\label{eq:density_crit}
\frac{d\mu_V}{dx} = \frac{b^2+1}{\pi b(x^2+b^2)(x^2+b^{-2})},
\qquad x \in \R,
\end{equation} and it can be checked that this is the limit of \eqref{eq:density_postcrit} as $a \to a_{cr}^+$.

The external field \eqref{eq:def_V} is weakly admissible,
since $V(x) - \frac{1}{2} \log(1+x^2)$ has a finite limit
as $|x| \to \infty$.  For admissible external fields this limit
is $+\infty$, and this is the main focus of the monograph 
\cite{SaffTotik1997} and many other works.
The study of weakly admissible external fields started with
Simeonov \cite{Simeonov2005},  see also the more recent
papers \cite{HardyKuijlaars2012}, \cite{Orive2019}.

The external fields \eqref{eq:def_V} belong to the external fields
studied in \cite{Orive2019}. The transition
at $a_{cr} = (b^2-1)^{-1}$ is contained in  
\cite[Theorem 3.14]{Orive2019}, where it is
shown (in a different, but equivalent, setup) 
that the support of $\mu_V$
is a bounded interval $[-A,A]$ if and only if $a > a_{cr}$
and the support is the full real line otherwise. In the latter
case, the equilibrium measure is a combination of balayage measures,
namely
\begin{equation} \label{eq:balayage} 
	\mu_V =	(1+a) \widehat{\delta}_{i b^{-1}} - a \widehat{\delta}_{ib} 
		\qquad \text{ in case } a \leq a_{cr},
	\end{equation}
where $\widehat{\delta}_{ic} = \Bal(\delta_{ic}, \R)$
is the balayage of the point mass at $ic$, $c > 0$,  to the real line.
It is known that this balayage has a (rescaled) Cauchy density
\[ d\widehat{\delta}_{ic}(x) = \frac{c}{\pi (x^2+c^2)} dx \]
and this indeed leads to the density \eqref{eq:density_precrit} for the
equilibrium measure $\mu_V$ in case $a \leq a_{cr}$.

For $a > a_{cr}$ the density of the right-hand side of \eqref{eq:balayage} 
becomes negative near infinity and it is no longer the equilibrium measure.

\subsection{Overview}
The rest of the paper is organized as follows.

The precritical case $a \leq a_{cr}$ is discussed in
section \ref{sec:proof_subcrit}.  Theorem \ref{thm:quadrature}
is a new result in this case as well, and it will be instructive 
to see from its proof how the equilibrium problem
with  external field $V$ arises in this case.

In section \ref{sec:Theorem16} we prove Theorem \ref{thm:equilibrium_muV} about the equilibrium measure
$\mu_V$ in particular for the case $a > a_{cr}$. We use
the technique of Schiffer variations that has been
used before, for example in \cite{MartinezRakhmanov2011}, but we
give a detailed exposition.

To prepare for the proof of  Theorems \ref{thm:main}
and \ref{thm:quadrature} in case $a > a_{cr}$, we first
collect a number of auxiliary results in section \ref{sec:preparations} about the transformation of logarithmic
potentials under stereographic projection.
It allows us to reformulate  Theorems \ref{thm:main}
and \ref{thm:quadrature} in terms of equalities and
inequalities for the logarithmic potentials of $\mu_V$
and a measure $\mu_{\Omega}$ on $\Omega$, see 
Proposition \ref{prop:UmuO}. 

In sections \ref{sec:equalities} and \ref{sec:inequalities}
we prove these equalities and inequalities.
The proof of equalities in section \ref{sec:equalities}
makes use of a function $S(z)$ that we construct out
of the equilibrium measure $\mu_V$ and that we call 
the spherical Schwarz function for the ellipse, since
$\partial \Omega$ will be characterized by the equation
\[ \partial \Omega : \quad \frac{\bar{z}}{1+|z|^2}
	= S(z), \]
see \eqref{eq:Omega_spherical_Schwarz}. The proof is by
a reduction to the usual Schwarz function, but it is largely
computational and we do not have a conceptual proof why the connection
should hold. However, it is the
key result that connects $\mu_V$ with the ellipse $\partial \Omega$.
The equalities of Proposition \ref{prop:UmuO} then follow
from standard calculations around Schwarz functions where
we use the complex Green's theorem
\begin{equation} \label{eq:Green}
	 \frac{1}{\pi} \int_{\Omega} f(z)dA(z) 
	= \frac{1}{2\pi i} \oint_{\partial \Omega}
		\frac{\partial f}{\partial \bar{z}}  dz,
		\end{equation}
where $dA$ denotes Lebesgue measure in the plane,
to transform integrals over $\Omega$ to integrals over 
its boundary.

In section \ref{sec:equalities} we prove the inequalities in Proposition \ref{prop:UmuO}. Here we use a dynamical picture 
that could be of independent interest. 
We analyze the way that $\mu_{\Omega}$ and $\mu_V$
evolve in terms of a time parameter $t = \frac{1}{1+2a}$.
Denoting the $t$-dependent quantities by $\Omega(t)$ and $V(t)$,
we find that both $(t\mu_{\Omega(t)})_t$ and $(t \mu_{V(t)})_t$
are increasing families of measures. By taking derivatives
with respect to $t$ we find measures $\rho_t = \frac{\partial}{\partial t}(t \mu_{\Omega(t)})$
and $\omega_t = \frac{\partial}{\partial t}(t \mu_{V(t)})$
that determine the evolution. This is similar to work
of Buyarov and Rakhmanov \cite{BuyarovRakhmanov1999}
who considered such derivatives for  equilibrium  measures
with varying masses. 

The measure $\rho_t$ is supported on $\partial \Omega(t)$
and describes the growth of the ellipse which is comparable
to Laplacian growth, see e.g.\ \cite{Gustafsson2014}.
In our situation the projections $\pm ib$ of the points
$p_1$ and $p_2$ act as repellers for the growth of the droplet
in a somewhat similar way as in \cite{Balogh2015} for usual
Laplacian growth. 

As in \cite{Balogh2015} our model also allows a natural discretization
that can be analyzed by polynomials that are orthogonal with
respect to a measure on the complex plane, and the orthogonality
can be rewritten as non-Hermitian orthogonality on a contour.
We plan to to come back to this in a separate publication.

\section{The case $a \leq a_{cr}$: proof of Theorem \ref{thm:quadrature}}
\label{sec:proof_subcrit}

While our main interest is in the case $a > a_{cr}$
we first discuss the case $a \leq a_{cr}$. 
Theorem~\ref{thm:equilibrium_muV} follows from the
results of \cite{Orive2019} as already discussed 
after the statement of the theorem. In particular the
measure $\mu_V$ is the combination of balayage measures
\eqref{eq:balayage}.

\begin{proof}[Proof of Theorem \ref{thm:quadrature} in
	case $a \leq a_{cr}$.]

In case $a \leq a_{cr}$ we know that 
$D = \S^2 \setminus \left( B_1 \cup B_2 \right)$
where $B_1$ and $B_2$ are spherical caps, centered at $p_1$ 
and $p_2$ respectively and
$\lambda(B_1) = \lambda(B_2) = \frac{a}{1+2a}$.
The geodesic radius of $B_1$ and $B_2$ is thus
\begin{equation}  \label{eq:def_ra} 
r_a = \arccos \left(1- \frac{2a}{1+2a} \right). 
\end{equation}

We start by noting that
\begin{align*} 
U^{\lambda_D}(x) & = \int_{\S^2 \setminus (B_1\cup B_2)}
\log \frac{1}{\|x-y\|} d\lambda(y) \\
& = \int_{\S^2 \setminus B_1} \log \frac{1}{\|x-y\|} d\lambda(y) 
- \int_{B_2} \log \frac{1}{\|x-y\|} d\lambda(y) \\
& = \int_{\S^2 \setminus \overline{B_1}} \log \frac{1}{\|x-y\|} d\lambda(y) 	- \int_{B_2} \log \frac{1}{\|x-y\|} d\lambda(y) 
\end{align*}
since $\lambda(\partial B_1) =  0$.
Then we recall the following mean value property for
the logarithmic potential (see for instance \cite[Proposition 3.2]{Beltran2003}): 
If $B = B(p,r)$ is an open spherical cap centered
at $p \in \S^2$ with geodesic radius $r \in (0, \pi)$, then
\begin{equation} \label{eq:mean_value} 
\int_B \log \frac{1}{\|x - y\|} d\lambda(y)
= \lambda(B) U^{\delta_p}(x) + c(r), \qquad x \in \S^2 \setminus B,
\end{equation}
where $c(r) = (1+ \cos(r)) \log \cos \frac{r}{2} 
- \frac{\cos r}{2} + \frac{1}{2}$ is a 
constant depending on $r$ only.
It also follows from \cite[Proposition 3.2]{Beltran2003}
(after a little calculation) that 
\begin{equation} \label{eq:mean_value2} 
\int_B \log \frac{1}{\|x - y\|} d\lambda(y)
< \lambda(B) U^{\delta_p}(x) + c(r), \qquad x \in B.
\end{equation}

We apply \eqref{eq:mean_value} to $B_2 = B(p_2,r_a)$ and 
\eqref{eq:mean_value}--\eqref{eq:mean_value2}
to $\S^2 \setminus \overline{ B_1} = B(-p_1, \pi - r_a)$,
where $-p_1$ is the antipodal point to $p_1$. It follows that
\begin{equation} \label{eq:UlambdaD_on_B1} 
U^{\lambda_D}(x) \begin{cases} = h_1(x),  & x \in B_1, \\
\leq h_1(x), & x \in D, \end{cases} \end{equation}
where
\begin{align} \nonumber  h_1(x) & = 
\lambda\left(\S^2 \setminus \overline{ B_1}\right)
U^{\delta_{-p_1}}(x) - \lambda(B_2) U^{\delta_{p_2}}(x)
+ c(\pi-r_a) - c(r_a), \\
& = \frac{1+a}{1+2a} U^{\delta_{-p_1}}(x)
- \frac{a}{1+2a} U^{\delta_{p_2}}(x)
+ c(\pi - r_a) - c(r_a). \label{eq:def_h1}
\end{align}
Similarly, 
\begin{equation} \label{eq:UlambdaD_on_B2} 
U^{\lambda_D}(x) \begin{cases} = h_2(x), & x \in B_2, \\
\leq h_2(x), & x \in D, \end{cases} \end{equation}
with
\begin{align}  h_2(x) 
& = \frac{1+a}{1+2a} U^{\delta_{-p_2}}(x)
- \frac{a}{1+2a} U^{\delta_{p_1}}(x)
+ c(\pi - r_a) - c(r_a). \label{eq:def_h2}
\end{align}

Let $\gamma \subset \S^2$ be the great circle 
containing all those points that are equidistant to $p_1$ and $p_2$.
Then $\gamma$ separates the sphere into two closed hemispheres
$H_1$ and $H_2$ where $H_1$ contains $p_1$ and $-p_2$ and $H_2$ contains $-p_1$ and $p_2$.
The balayage of a measure $\sigma$ supported on one of the hemispheres
onto $\gamma$ is  the unique measure $\widehat{\sigma}$ supported on $\gamma$ 
with $m(\widehat{\sigma}) = m(\sigma)$ such
that $U^{\widehat{\sigma}} - U^{\sigma}$ is constant on
the other hemisphere. Thus, if  
\begin{equation} \label{eq:eta_balayage} 
\eta = \frac{1+a}{1+2a} \widehat{\delta}_{-p_2} -
\frac{a}{1+2a} \widehat{\delta}_{p_1} 
\end{equation}
then it follows from \eqref{eq:UlambdaD_on_B1} and \eqref{eq:def_h1} that, for some constant $m$,
\begin{equation} \label{eq:Ueta1} 
U^{\lambda_D}(x)  \begin{cases} 
= U^{\eta}(x) + m, & x \in B_1, \\
\leq U^{\eta}(x) + m, & x \in H_1 \cap D.
\end{cases}
\end{equation}
Because of symmetry ($p_1$ and $p_2$ lie symmetric with respect to $\gamma$,
and so do $-p_1$ and $-p_2$, which gives $\widehat{\delta}_{\pm p_1} = 
\widehat{\delta}_{\pm p_2}$), we  also have
\[ \eta = \frac{1+a}{1+2a} \widehat{\delta}_{-p_1} -
\frac{a}{1+2a} \widehat{\delta}_{p_2} \]
and by \eqref{eq:UlambdaD_on_B2} and \eqref{eq:def_h2} 
\begin{equation} \label{eq:Ueta2} 
U^{\lambda_D}(x) \begin{cases} 
= U^{\eta}(x) + m, &  x \in B_2, \\
\leq U^{\eta}(x) + m, &  x \in H_2 \cap D.
\end{cases} 
\end{equation}
with the same constant $m$.

Then putting
\begin{equation}  \label{eq:sigmastar_ansatz}
\sigma^* = \frac{1+2a}{2a} \eta \end{equation}
we see from \eqref{eq:Ueta1} and \eqref{eq:Ueta2} that \eqref{eq:equality_potentials}
and \eqref{eq:inequality_potentials} are satisfied.
Applying the stereographic projection to \eqref{eq:eta_balayage} 
we find, since $\phi(-p_2) = i b^{-1}$ and $\phi(p_1) = i b$,
and $\gamma$ is mapped to $\R$,
\[ (1+2a) \phi_*(\eta) = (1+a)
\Bal\left(\delta_{i b^{-1}}, \R \right)
- a 		\Bal\left(\delta_{i b}, \R\right) \]
which according to  \eqref{eq:balayage} is equal to the equilibrium
measure $\mu_V$, since $a \leq a_{cr}$.
Then \eqref{eq:sigmastar} follows, and since $\mu_V$ is a probability measure it
also follows that $\sigma^*$ is a positive measure supported on $\gamma =
\phi^{-1}(\R \cup \{\infty\})$ with total mass $\frac{1}{2a}$.

This completes the proof of Theorem \ref{thm:quadrature}  
in the case $a \leq a_{cr}$.
\end{proof}

\section{Proof of Theorem \ref{thm:equilibrium_muV}}
\label{sec:Theorem16}

Theorem \ref{thm:equilibrium_muV} is proved by Orive, S\'anchez Lara, and Wielonsky 
\cite{Orive2019} in case $a \leq a_{cr}$, as already noted.
They also showed that for $a > a_{cr}$ the support of $\mu_V$ is an
interval $[-A,A]$. 
Our task here is to compute $A$ and prove that 
the density of $\mu_V$ on $[-A,A]$ is equal to \eqref{eq:density_postcrit}.

The external field \eqref{eq:def_V} extends to an analytic function in 
a neighborhood of the real line that is also denoted by $V$.
Its derivative is a rational function that we consider in
the full complex plane.

The following lemma is not new. It is a variation of Theorem 1.34  
in \cite{Deift1998} which deals with equilibrium problems on a finite
interval instead of the full real line.
The proof in \cite{Deift1998} is based on an analysis of weighted Fekete points. 
The proof below uses the technique of Schiffer variations, which was used
for example in \cite{MartinezOriveRakhmanov2015} for equilibrium
problems on the real line. The technique is also useful for
equilibrium problems on contours in the complex plane
\cite{MartinezRakhmanov2011}, \cite{KuijlaarsSilva2015} as a way to
construct contours with the so-called $S$-property \cite{GoncharRakhmanov1987}, see also \cite{MartinezRakhmanov2016}
and references therein.
Since we feel it deserves to be better known we give a 
self-contained exposition. 

\begin{lemma}\label{lem:rational} 
	With  
	\begin{equation} \label{eq:def_R}
	R(z) :=  \left[V'(z) \right]^2-2\int_\R\frac{V'(z)-V'(s)}{z-s}d\mu_V(s)
	\end{equation}
	the following hold:
	\begin{enumerate}
		\item[\rm (a)] $R$ is a rational function with double poles at $\pm ib$ and $\pm i b^{-1}$ and 
		\begin{equation}\label{eq:schiffer03}
		R(z) = \left[\int_\R\frac{d\mu_V(s)}{z-s}-V'(z)\right]^2, \qquad z \in \C \setminus \supp(\mu_V).
		\end{equation}
		\item[\rm (b)] The equilibrium measure has support
		\begin{equation}\label{eq:sup_muv}
		\supp \mu_V = \overline{\{x\in\R : R(x) < 0\}}
		\end{equation} and density
		\begin{equation}\label{eq:density_muv}
		\frac{d\mu_V(x)}{dx} = \frac{1}{\pi}\sqrt{R^-(x)}, \qquad x\in \supp \mu_V,
		\end{equation} where $R^- = \max(0,-R)$	denotes the negative part of $R$.
	\end{enumerate}
\end{lemma}

\begin{proof} (a) By \eqref{eq:def_V} 	 we have
	\begin{equation}  \label{eq:Vprime}
	V'(z) = \frac{1+a}{z^2+b^{-2}} - \frac{a}{z^2 + b^2} 
	\end{equation} 
	and then it is clear from \eqref{eq:def_R} that $R$ is a rational function with double poles 
	at $z = \pm ib$ and $z = \pm ib^{-1}$. 
	To prove \eqref{eq:schiffer03}, we apply Schiffer variations which goes as follows. 
	Let $h:\R\to\R$ be a bounded $C^2$ 
	function and following \cite[Lemma 3.1]{MartinezRakhmanov2011},
	we  consider the family of measures $\{\mu_\eps\}_{\eps \in \R}$, 
	given by their action on a continuous function $f$ by
	\[
	\int_\R f(s)d\mu_\eps(s) = \int_\R f(s+\eps h(s))d\mu_0(s).
	\] 
	Since $\mu_0 = \mu_V$ and each $\mu_{\eps}$ is a probability measure  it follows that
	$\eps \mapsto I_V[\mu_{\eps}]$ has its minimum at $\eps = 0$. 
	Observe that
	\begin{align*}
	\iint \log\frac{1}{|t-s|}d\mu_\eps(t)d\mu_\eps(s) &= -\iint \log|t-s+\eps[h(t)-h(s)]|d\mu_V(t)d\mu_V(s)\\
	&= -\iint  \left(\log |t-s|+\log\left|1+\eps\frac{h(t)-h(s)}{t-s}\right|\right)d\mu_V(t)d\mu_V(s)\\
	&= \iint \log\frac{1}{|t-s|}d\mu_V(t)d\mu_V(s)\\
	&\quad-\eps\iint_\R\frac{h(t)-h(s)}{t-s}d\mu_V(t)d\mu_V(s)+o(\eps),
	\end{align*} as $\eps\to 0$, and
	\begin{align*}
	2\int V(s)d\mu_\eps(s) &= \int_\R V(s+\eps h(s))d\mu_V(s)\\
	&= 2\int \left(V(s)+\eps h(s)V'(s)+O(\eps^2)\right)d\mu_V(s)\\
	&= 2\int V(s)d\mu_V(s)+2\eps\int_\R h(s)V'(s)d\mu_V(s) + o(\eps),
	\end{align*} as $\eps\to 0$. Since $I[\mu_{\eps}] + 2 \int V d\mu_{\eps}$ has
	its minimum at $\eps =0$, it follows from the above that
	\begin{equation}\label{eq:schiffer01}
	\iint \frac{h(t)-h(s)}{t-s}d\mu_V(t)d\mu_V(s) = 2\int h(s)V'(s)d\mu_V(s).
	\end{equation}
	By taking real and imaginary parts separately, the identity \eqref{eq:schiffer01} 
	is also satisfied for a complex valued $h:\R\to\C$. In particular, taking
	\[
	h(s) = \frac{1}{z-s}
	\] for some $z\in\C\setminus\R$, we have that
	\begin{align} \nonumber  
	\left(\int \frac{d\mu_V(s)}{z-s}\right)^2 & = 2 \int \frac{V'(s)}{z-s}d\mu_V(s) \\
	&= 2 V'(z) \int \frac{d\mu_V(s)}{z-s} - 2 \int \frac{V'(z)-V'(s)}{z-s} d\mu_V(s).
	\label{eq:schiffer02}
	\end{align} 
	After bringing the first term on the right to the left, and completing the square
	we obtain \eqref{eq:schiffer03} in view of the definition \eqref{eq:def_R} of $R$. 
	
	\medskip
	
	(b) Observe from the square in \eqref{eq:schiffer03} 
	that $R(x) \geq 0$ for every $x \in \R \setminus \supp(\mu_V)$, and therefore we certainly have
	\[ \supp(\mu_V) \subset \overline{\{ x \in \R :
		R(x) < 0\}}. \]
	
	Let us denote by
	\[
	F_{\mu}(z) = \int_\R \frac{d\mu(s)}{z-s}
	\] 
	the Stieltjes transform of a measure $\mu$. If $z$ is in a neighborhood of the real line where $V(z)$ is analytic, then \eqref{eq:schiffer03} reads
	\begin{equation} \label{eq:schiffer04}
	[F_{\mu_V}(z)-V'(z)]^2 = R(z),
	\end{equation}
	from which we obtain
	\begin{equation}\label{eq:stieltjes01}
	F_{\mu_V}(z) = V'(z) -R(z)^{1/2},
	\end{equation} 
	where $R(z)^{1/2}$ denotes the appropriate analytic square root of $R(z)$.  
	Since $\mu_V$ has no atoms and $V'(z)$ is real for real $z$, we can apply the Perron--Stieltjes inversion formula to \eqref{eq:stieltjes01} for an interval $[x_1,x_2]$ with $x_1 < x_2$. 
	We obtain that
	\begin{align}\label{eq:stieltjes02}
	\mu_V([x_1,x_2]) &= 
	\lim_{\delta\to 0^+} \int_{x_1}^{x_2}\Im F_{\mu_V}(x-i\delta) ^{1/2}\,dx\nonumber\\
	& =
	\lim_{\delta\to 0^+}\int_{x_1}^{x_2}\text{Im}\,R(x-i\delta)^{1/2}\,dx\nonumber\\
	&= \frac{1}{\pi}\int_{x_1}^{x_2} \Im R_-^{1/2}(x)\,dx.
	\end{align} 
	Only the values $x\in [x_1,x_2]$ with $R(x)<0$ contribute to the integral in \eqref{eq:stieltjes02}, 
	which gives \eqref{eq:sup_muv} and then \eqref{eq:density_muv} follows from
	\eqref{eq:stieltjes02} as well.
\end{proof}

We are now ready for the proof of Theorem \ref{thm:equilibrium_muV}.

\begin{proof}[Proof of Theorem \ref{thm:equilibrium_muV}.] 
	
	We consider $a > a_{cr}$, since the case $a \leq a_{cr}$
	already follows from \cite{Orive2019}, as already noted.
	
	From \eqref{eq:def_R} and \eqref{eq:Vprime}  we obtain $R(z) = O(z^{-4})$ as $z\to \infty$. 
	Since $R$ has double poles at $\pm i b$ and $\pm i b^{-1}$ (and no other poles), we  
	conclude that
	\[
	R(z) = \frac{pz^4+qz^2+r}{(z^2+b^2)^2(z^2+b^{-2})^2}
	\] 
	for certain real coefficients $p$, $q$ and $r$. Thus $R$ has at most four zeros in $\C$.
	We are in the situation that the support of $\mu_V$ is a finite interval $[-A,A]$.
	We then see from  \eqref{eq:sup_muv} that $R$ has a sign change at $\pm A$, 
	which means that $R$ has a zero of odd order at $\pm A$. The zeros are actually simple,
	since there are no more than four zeros in total (counting  multiplicities), and by
	symmetry, the orders of the zeros in $A$ and $-A$ are the same. 
	
	There are at most two remaining zeros of $R$ in $\C$. Any zero in $\C \setminus \supp(\mu_V)$
	has even order because of the square in the formula \eqref{eq:schiffer03}. By symmetry,
	with any such zero, its negative and its complex conjugate are also zeros of the same
	order.  So there are no zeros in $\C \setminus \supp(\mu_V)$. 
	Since 
	\[ \supp(\mu_V) = [-A,A] = \overline{\{ x \in \R : R(x) < 0 \}} \]
	any zero in $(-A,A)$ must be of even order. This rules out any additional zero,
	except maybe a double zero at $z=0$. But then by \eqref{eq:density_muv}
	the equilibrium density behaves as $\sim c |x|$ as $x \to 0$, for some $c > 0$,
	and this is not possible. Indeed, if the support of the equilibrium
	measure is known then the density can be recovered by solving a
	singular integral equation and from the explicit formulas as
	in \cite[Chapter IV, Theorem 3.2]{SaffTotik1997} one sees that the
	density is real analytic in the interior of its support,
	for a real analytic external field.

	Thus, apart from $-A$, there are no other zeros of $R$ in $\C$, and it follows that $p =0$,
	and
	\begin{equation} \label{eq:schiffer05}
	R(z) = \frac{C(z^2-A^2)}{(z^2+b^2)^2(z^2+b^{-2})^2}
	\end{equation}
	for certain positive constant $C > 0$. From \eqref{eq:schiffer04}
	we find, 
	\begin{equation} \label{eq:schiffer06} 
	\begin{aligned}
	\lim_{z \to \pm ib} (z^2+b^2)^2 R(z) & =
	\lim_{z \to \pm ib} (z^2+b^2)^2 V'(z)^2, \\
	\lim_{z \to \pm ib^{-1}} (z^2+b^{-2})^2 R(z) & =
	\lim_{z \to \pm ib^{-1}} (z^2+b^{-2})^2 V'(z)^2,
	\end{aligned}
	\end{equation}
	which in view of \eqref{eq:schiffer05} and \eqref{eq:Vprime} 
	leads to the equations
	\begin{align*}
	\frac{C(-b^2-A^2)}{(-b^2+b^{-2})^2} & = -a^2b^2, \\
	\frac{C(-b^{-2}-A^2)}{(-b^{-2}+b^2)^2} &= -(1+a)^2 b^{-2}.
	\end{align*} 
	This simplifies to two equations for $C$ and $A^2$, namely
	\begin{align*}
	Cb^4+CA^2b^2 &= a^2(b^4-1)^2,\\
	Cb^4+CA^2b^6 &= (1+a)^2(b^4-1)^2,
	\end{align*} 
	with solution \eqref{eq:AC_postcrit}, since $A > 0$.
	This completes the proof of Theorem \ref{thm:equilibrium_muV} in case $a > a_{cr}$.
\end{proof}

\section{Preparation for the proofs of Theorem \ref{thm:main} and
	\ref{thm:quadrature} in case $a > a_{cr}$}
\label{sec:preparations}

\subsection{Stereographic projection}

The pushforward $\phi_*(\mu)$ of a measure $\mu$ on $\S^2$
is the measure on $\C \cup \{\infty\}$ that 
assigns to a Borel set $B$ the value $\phi_*(\mu)(B) = \mu(\phi^{-1}(B))$.
The pushforward of the normalized Lebesgue measure is
\begin{equation} \label{eq:phi_lambda} 
	\phi_*(\lambda) =  \frac{dA(z)}{\pi(1+|z|^2)^2} 
	\end{equation}
where $dA$ denotes the two dimensional 
Lebesgue measure in the plane, and so for every $D$
\begin{equation} \label{eq:phi_lambdaD}
	\phi_*(\lambda_D) =  \left. \frac{dA(z)}{\pi(1+|z|^2)^2} \right|_{\Omega}, \qquad \text{where } \Omega = \phi(D). 
	\end{equation}

We will need to know how logarithmic potentials and weighted
logarithmic energies transform under the stereographic
projection. This can be computed from the  basic formula  for the  transformation 
of distances. If $x, y \in \S^2$, $z = \phi(x)$, $w = \phi(y)$, then
\[ \| x- y\| = \frac{2 |z - w|}{\sqrt{1+|z|^2} \sqrt{1+|w|^2}}, \]
with proper modification in case $z$ or $w$ is at infinity.
For a measure $\mu$ on $\S^2$ with pushforward measure $\phi_*(\mu)$ we thus have if $z = \phi(x)$,
\begin{align} \nonumber
	U^{\mu}(x) 
		&	= \int \log \frac{\sqrt{1+|z|^2}\sqrt{1+|w|^2}}{2 |z-w|}
			d\phi_*(\mu)(w) \nonumber \\
		& = U^{\phi_*(\mu)}(z)
		+ \frac{\m(\mu)}{2}  \log (1+|z|^2)
			+ \frac{1}{2} \int \log(1+|w|^2) d\phi_*(\mu)(w)
			- \m(\mu) \log 2,
			\label{eq:phi_Umu}
		\end{align}
and
\begin{align} \nonumber
I[\mu] & = \iint \log \frac{\sqrt{1+|z|^2} \sqrt{1+|w|^2}}{2|z-w|} d\phi_*(\mu)(z) d\phi_*(\mu)(w) \nonumber \\
& = I[\phi_*(\mu)] + \m(\mu) \int \log( 1+|z|^2) d\phi_*(\mu)(z) - \m(\mu)^2 \log 2. 
	\label{eq:phi_Imu}
\end{align}

Hence, in the presence of an external field $Q$ we have
(for a probability measure $\mu$)
\begin{equation} \label{eq:phi_IQmu} 
	I_Q[\mu]  = I_{\widehat{Q}}[\phi_*(\mu)] - \log 2 
	\end{equation}
with
\begin{equation}  \label{eq:def_Qstar} 
	\widehat{Q}(z) = Q(\phi^{-1}(z)) + \frac{1}{2} \log(1+ |z|^2).
	\end{equation}
If $\mu$ is the equilibrium measure with external field
$Q$ on $\S^2$ then $\phi_*(\mu)$ is the equilibrium
measure with external field $\widehat{Q}$ on $\C$. Conversely,
if $\mu_V$ is the equilibrium measure with external
field $V$ on $\C$ (or on $\R$) then $(\phi^{-1})_*(\mu_V)$
is the equilibrium measure with external field
\[ V \circ \phi   - \frac{1}{2} \log\left(1 + \phi^2\right)
	\quad \text{on } \S^2 \quad \text{ (or on $\phi^{-1}(\mathbb R)$).} \]
	
\subsection{Transformation of the theorems to $\C$}

Since we prefer to do the calculations in $\C$ rather than on
$\S^2$, we first transform Theorems~\ref{thm:main}
and \ref{thm:quadrature} to statements about logarithmic
potentials in the complex plane.
We continue with the case $a >a_{cr}$.

Let $\partial \Omega$ be the ellipse from Theorem \ref{thm:main}.
and define $\mu_\Omega$  on $\Omega$ by
\begin{equation} \label{eq:muOmega} 
d\mu_\Omega(z) := \left. \frac{1+2a}{\pi}\frac{dA(z)}{(1+|z|^2)^2}\right|_\Omega,
\end{equation}
where $dA(z)$ is the usual area measure on $\C$. 

Theorems \ref{thm:main} and \ref{thm:quadrature} for $a > a_{cr}$ then follow from the following
\begin{proposition} \label{prop:UmuO}
		Let $a >  a_{cr}$. Then
	$\mu_{\Omega}$ is a probability measure on $\Omega$
	whose logarithmic potential satisfies
	\begin{equation} \label{eq:UmuO_outside} 
	U^{\mu_{\Omega}}(z)
	\begin{cases} = U^{\mu_V}(z), & z \in \C \setminus \Omega, \\
	\leq U^{\mu_V}(z),  & z \in \C, 
	\end{cases}
	\end{equation}
	and, with some constant $c$,
	\begin{equation} \label{eq:UmuO_inside}
	U^{\mu_{\Omega}}(z) + a \log \frac{1}{|z^2+b^2|}
	+ \frac{1+2a}{2} \log \left(1+|z|^2 \right) 
	\begin{cases} = c, & z \in \Omega, \\
	\geq c, & z \in \C. 
	\end{cases}
	\end{equation}
\end{proposition}

The proof of Proposition \ref{prop:UmuO} is in sections 5 and 6 below.

To see how Theorems \ref{thm:main} and \ref{thm:quadrature}
follow from Proposition \ref{prop:UmuO}, we put $D = \phi^{-1}(\Omega)$. Under inverse
stereographic projection, the measure
$\mu_{\Omega}$ transforms by \eqref{eq:phi_lambdaD} and \eqref{eq:muOmega} 
to
\begin{equation} \label{eq:phiinv_muO} 
	(\phi^{-1})_*(\mu_{\Omega}) = (1+2a) \lambda_D.
	\end{equation}
Then $\lambda(D) = (1+2a)^{-1}$ since according to
Proposition \ref{prop:UmuO} $\mu_V$ is a probability measure
and then  \eqref{eq:phiinv_muO} is a probability
measure as well.

By \eqref{eq:phi_Umu} we have if $x \in \S^2$ and $z = \phi(x)$,
\begin{equation} \label{eq:UlambdaD} 
	(1+2a) U^{\lambda_D}(x) =
	U^{\mu_{\Omega}}(z) 
		+ \frac{1}{2} \log(1+|z|^2) + C_1 \end{equation}
for some constant $C_1$, and also 
since $\phi_*(\sigma) = a (\delta_{ib} + \delta_{-ib})$, 
\begin{align} \nonumber  U^{\sigma}(x) & = a U^{\delta_{ib} + \delta_{-ib}}(z)
	+ a \log (1+|z|^2) + C_2 \\
	\label{eq:Usigma}
	& = a \log \frac{1}{|z^2+b^2|} + a \log (1+|z|^2) + C_2,
\end{align}
with another constant $C_2$.
Combining \eqref{eq:UlambdaD} and \eqref{eq:Usigma} with
 \eqref{eq:UmuO_inside} we obtain
 \eqref{eq:frostman_lambdaD} which shows that $\mu_{\sigma} = 
 (1+2a) \lambda_D$ and so indeed $D = \phi^{-1}(\Omega)$ is the droplet as claimed in Theorem~\ref{thm:main}.

\medskip

Theorem \ref{thm:quadrature} follows similarly from \eqref{eq:UmuO_outside}. 
We use \eqref{eq:sigmastar} to define $\sigma^*$ in terms
of $\mu_V$. Then $\sigma^*$ is a measure on $\phi^{-1}(\R)$ with total mass
$\m(\sigma) = \frac{1}{2a}$. Transforming the equality
and inequality in \eqref{eq:UmuO_outside} back to the sphere by
means of \eqref{eq:phi_Umu} we obtain 
\eqref{eq:equality_potentials} and \eqref{eq:inequality_potentials}
as required for Theorem \ref{thm:quadrature}.

\section{Proof of the equalities in Proposition \ref{prop:UmuO}}
\label{sec:equalities}

In this section we prove the equalities from
Proposition \ref{prop:UmuO}. The inequalities will be dealt
with in the next section.
	
\subsection{Spherical Schwarz function}

For the proof of the equalities we make use of the function
\begin{equation} \label{eq:spherical_Schwarz} 
	S(z) = \frac{1}{1+2a} \left[ \frac{2az}{z^2+b^2} + \int \frac{d\mu_V(s)}{z-s} \right] \end{equation}
that we call the \emph{spherical Schwarz function} because of the following property.

\begin{proposition} 
	The ellipse $\partial \Omega$ is characterized by the equation
	\begin{equation} \label{eq:Omega_spherical_Schwarz} 
	\partial \Omega : \quad \frac{\bar{z}}{1+|z|^2} = S(z) 
	\end{equation}
\end{proposition}
\begin{proof}
	If in the equation for an ellipse 	$\frac{x^2}{p^2} + \frac{y^2}{q^2} = 1$, $p > q$, 
	we write $x = \frac{z+\bar{z}}{2}$, $y = \frac{z- \bar{z}}{2i}$, 
	and solve for $\bar{z}$, then we obtain
	\begin{equation} \label{eq:Omega_Schwarz} 
		 \bar{z} = S_0(z), \quad S_0(z) = \frac{(p^2+q^2)z - 2pq (z^2-r^2)^{1/2}}{r^2}, \end{equation}
	where $r^2 = p^2 - q^2$, as the equation for the ellipse. 
	Thus $S_0$ is the well-known usual Schwarz function for
	the ellipse. For the ellipse \eqref{eq:equation_ellipse} 
	from Theorem \ref{thm:main} we have
	\begin{align} \label{eq:Omega_parameters} 
	p^2  =   \frac{b^2+1}{2(b^2a-a-1)}, \qquad q^2  = \frac{b^2-1}{2(b^2a+a+1)}, \qquad r^2 = A^2,
	\end{align}
	with  $A$ as in \eqref{eq:AC_postcrit}. 	
	
	From \eqref{eq:Omega_Schwarz} we find
	\[ \partial \Omega : \quad \frac{\bar{z}}{1+|z|^2} = \frac{S_0(z)}{1+z S_0(z)} \]
	and to obtain \eqref{eq:Omega_spherical_Schwarz} we will have
	to verify that 
	\begin{equation} \label{eq:Schwarz_identity} 	\frac{S_0(z)}{1+zS_0(z)} = S(z). 
	\end{equation}
	The identity \eqref{eq:Schwarz_identity} can be checked by straightforward calculations.
	By  \eqref{eq:spherical_Schwarz}, \eqref{eq:stieltjes01}, \eqref{eq:Vprime}, 	and \eqref{eq:schiffer05}, we have
	\begin{align} \nonumber
	(1+2a) S(z) & = \frac{2az}{z^2+b^2} + V'(z) - R(z)^{1/2} \\
		& = \frac{az}{z^2+b^2} + \frac{(1+a)z}{z^2 + b^{-2}} - 
		\frac{\sqrt{C} (z^2-A^2)^{1/2}}{(z^2+b^2)(z^2+b^{-2})}.
		\label{eq:spherical_Schwarz_meromorphic}
	\end{align} 
	
	Hence both $S(z)$ and $S_0(z)$ are of the form  $P(z) + Q(z) (z^2-A^2)^{1/2}$ 
	with rational $P$ and $Q$.
	To prove \eqref{eq:Schwarz_identity} we expand $(1+zS_0(z)) S(z)$ into this form 
	and verify that it is equal to $S_0(z)$ from \eqref{eq:Omega_Schwarz} with parameters
	\eqref{eq:Omega_parameters}. We also need that  $A$ and $C$ are 
	given by \eqref{eq:AC_postcrit}. It all fits rather miraculously, and
	we obtain \eqref{eq:Schwarz_identity} and then also the proposition.
\end{proof}

\subsection{About the identity  \eqref{eq:Schwarz_identity}}
	
	The identity \eqref{eq:Schwarz_identity} is rather surprising, and indeed it
	is more special than it may seem at first sight. The identity implies of course
	that the zeros and poles of both sides are the same. From \eqref{eq:spherical_Schwarz_meromorphic}
	it follows that $S$ has poles at $z= \pm i b$, while poles of the left-hand side of
	\eqref{eq:Schwarz_identity} can only come from zeros of $z \mapsto 1+z S_0(z)$.
	It is not immediate that $z = \pm ib $ is a zero of $z \mapsto 1 + z S_0(z)$, but it
	can be verified by explicit calculation.
	
	Extending this idea, we can in fact prove \eqref{eq:Schwarz_identity} by
	examining the zeros and poles of $S(z)$, $S_0(z)$ and $1+zS_0(z)$ 
	on the two sheeted Riemann  surface associated with $w^2 = z^2-A^2$.
	Both $S$ and $S_0$ have  analytic continuation to the
	second sheet of this Riemann surface, just by taking the different sign of the 
	square roots in
	\eqref{eq:Omega_Schwarz} and \eqref{eq:spherical_Schwarz_meromorphic}.
	Then the following can be checked:
	\begin{itemize}
	\item[(1)] $S_0$ has simple zeros at  the points
	$z = \pm i \frac{2pq }{r}$ on the first sheet of the Riemann
	surface, simple poles at the two points at infinity, 
	and no other zeros or poles.
	\item[(2)] $z \mapsto 1 + z S_0(z)$ has	
	double poles at the two points at infinity, four simple
	zeros at the points
	\begin{equation} 
	\begin{aligned} 
	\pm i \frac{\sqrt{p^2+q^2+2p^2q^2+2pq \sqrt{1+p^2}\sqrt{1+q^2}}}{r} & \quad \text{ on the first sheet}, \\
	\pm i \frac{\sqrt{p^2+q^2+2p^2q^2-2pq \sqrt{1+p^2}\sqrt{1+q^2}}}{r}
	 &  	\quad \text{ on the second sheet},
	\end{aligned} \label{eq:Spoles}
	\end{equation}
	and no other zeros or poles.
	\item[(3)] $S$ has four simple zeros and four simple poles.
	The poles are at $\pm i b$ on the first sheet
	and $\pm i b^{-1}$ on the second sheet. The zeros are at 
	\[ \pm i \frac{\sqrt{b^4-1}}{b \sqrt{1+2a}}
		\quad \text{ on the first sheet}, \]
	and at the two points at infinity.
	\end{itemize}

	It follows from (1) and (2) that
	\begin{itemize}
		\item[(4)] $z \mapsto \frac{S_0(z)}{1+zS_0(z)}$
		has simple zeros at $\pm i \frac{2pq}{r}$ on the
		first sheet, two simple zeros at the points
		at infinity, and four simple poles at the points
	in \eqref{eq:Spoles}.
	\end{itemize}
	This agrees with the zeros and poles from (3) provided that
    \begin{align*}
    \frac{\sqrt{b^4-1}}{b \sqrt{1+2a}} & = \frac{2pq}{r}, \\
	 b^2 & =  \frac{p^2+q^2+2p^2q^2+2pq \sqrt{1+p^2}\sqrt{1+q^2}}{r^2}, \\
     b^{-2} & = \frac{p^2+q^2+2p^2q^2-2pq \sqrt{1+p^2}\sqrt{1+q^2}}{r^2},
    \end{align*}
    and these identities are indeed consequences of the formulas \eqref{eq:Omega_parameters}
    for $p$, $q$ and $r$.
    
    Thus both sides of \eqref{eq:Schwarz_identity} are meromorphic functions
    on the compact Riemann surface with the same zeros and poles, and as a result
    their ratio is a constant. The constant is one, since 
    both functions behave as $z^{-1} + \mathcal{O}(z^{-2})$
    as $z \to \infty$, as is easy to check from
    \eqref{eq:spherical_Schwarz} and \eqref{eq:Omega_Schwarz}.

\subsection{Proof of the equalities}

\begin{lemma} \label{lem:FmuO_equality} 
	$\mu_\Omega$ is a probability measure on $\Omega$ whose Stieltjes transform satisfies
	\begin{align} \label{eq:muOmega_Stieltjes1}
		\int_{\Omega} \frac{d\mu_{\Omega}(s)}{z-s}	
		 & = \int_{\R} \frac{d\mu_V(s)}{z-s}, &&
		 z \in \C \setminus \Omega, \\ \label{eq:muOmega_Stieltjes2}
		\int_{\Omega} \frac{d\mu_{\Omega}(s)}{z-s} & =
		 - \frac{2az}{z^2+b^2} + \frac{(1+2a) \bar{z}}{1+|z|^2}, 
		 && z \in \Omega. 
		 \end{align}
\end{lemma}
\begin{proof} 
Consider $z\in\C\setminus \Omega$ first. Then, by the definition \eqref{eq:muOmega} of $\mu_\Omega$ 
and Green's Theorem in the complex plane \eqref{eq:Green},
	\begin{align*}
		\int_\Omega\frac{d\mu_\Omega(s)}{z-s} &= \frac{1+2a}{\pi}\int_\Omega\frac{dA(s)}{(z-s)(1+|z|^2)^2} \\
		&= \frac{1+2a}{2\pi i}\oint_{\partial\Omega}\frac{\bar{s}}{(z-s)(1+|s|^2)}ds.
	\end{align*} 
Here we use the property \eqref{eq:Omega_spherical_Schwarz}
of the spherical Schwarz function, and we find
\begin{align}\label{eq:eq_lema_52}
	\int_\Omega\frac{d\mu_\Omega(s)}{z-s} = 
	\frac{1+2a}{2\pi i}\oint_{\partial\Omega}\frac{S(s)}{z-s} ds
\end{align} 
to which we apply the Residue Theorem for the exterior region $\C\setminus\Omega$. 
The spherical Schwarz function \eqref{eq:spherical_Schwarz} 
has simple poles at $s=\pm ib$ (which are in the exterior of $\Omega$) with residues $\frac{a}{1+2a}$, 
and they give the contribution 
$-\frac{a}{z - i b} - \frac{a}{z+ib} = - \frac{2az}{z^2+b^2}$
to the integral \eqref{eq:eq_lema_52}. 
	
	The contribution from the pole at $s = z$ is $(1+2a)S(z)$, and therefore
	\[ \int_{\Omega} \frac{d\mu_{\Omega}(s)}{z-s}
		= - \frac{2az}{z^2+b^2} + (1+2a)S(z), \]
	since there is no contribution from infinity.
	Then \eqref{eq:muOmega_Stieltjes1} follows because
	of \eqref{eq:spherical_Schwarz}.

	Letting $z \to \infty$ in \eqref{eq:muOmega_Stieltjes1} we find
	$\m(\mu_{\Omega}) = \m(\mu_V) = 1$, and therefore 	
  	 $\mu_{\Omega}$ is a probability measure, as it is clearly positive from \eqref{eq:muOmega}.
It remains to prove \eqref{eq:muOmega_Stieltjes2}.

Let $z\in\Omega\setminus\partial\Omega$ and let $r>0$ be such that the open disk $D_r(z) = \{w\in\C : |z-w|<r\}$ is contained in $\Omega\setminus\partial\Omega$. Then, by the definition of
$\mu_{\Omega}$ in \eqref{eq:muOmega} and the complex Green's Theorem \eqref{eq:Green},
	\begin{align*}
		\int_{\Omega\setminus D_r(z)}\frac{d\mu_\Omega(s)}{z-s} &
		= \frac{1+2a}{\pi}\int_{\Omega\setminus D_r(z)}\frac{dA(s)}{(z-s)(1+|s|^2)^2}\\
		&= \frac{1+2a}{2\pi i}\oint_{\partial\Omega}\frac{\bar{s}}{(z-s)(1+|s|^2)}ds-
		\frac{1+2a}{2\pi i}\oint_{\partial D_r(z)}\frac{\bar{s}}{(z-s)(1+|s|^2)}ds. 
	\end{align*} 
The integral over $\partial \Omega$ is evaluated using the spherical Schwarz function 
and the Residue Theorem for the exterior region as in the proof 
of \eqref{eq:muOmega_Stieltjes1}, but now 
	there is no contribution from the pole at $s=z$. We find
	\begin{align*}
	\frac{1+2a}{2\pi i}\oint_{\partial\Omega}\frac{\bar{s}}{(z-s)(1+|s|^2)}ds	&  =
	\frac{1+2a}{2\pi i} \oint_{\partial\Omega} \frac{S(s)}{z-s} ds = - \frac{2az}{z^2+b^2}
	\end{align*}
	from the residues at $s=\pm i b$.
	In  the integral over the circle $\partial D_r(z)$ we write $s = z + r e^{i\theta}$, 
	\begin{align*} 
		- \frac{1+2a}{2\pi i}	\oint_{\partial D_r(z)} \frac{\bar{s}}{(z-s)(1+|s|^2)} ds
		& = \frac{1+2a}{2\pi} 	\int_0^{2\pi} \frac{\bar{z} + r e^{-i\theta}}
			{1+ |z+ r e^{i \theta}|^2} d\theta \\
		& = \frac{(1+2a) |z|}{1+|z|^2} + \mathcal{O}(r) \qquad \text{as } r \to 0^+.
		\end{align*}
	Letting $r \to 0^+$ we  find \eqref{eq:muOmega_Stieltjes2} for $z \in \Omega 
	\setminus \partial \Omega$, and then by continuity also for $z \in \partial \Omega$.
	\end{proof}

After integration we obtain from Lemma \ref{lem:FmuO_equality}
the desired equalities.
In the proof we use that for any measure $\mu$ on $\C$ one has
\begin{equation} \label{eq:Umu-zderivative} 
	- 2 \frac{\partial}{\partial z} U^{\mu}(z) = \int \frac{d\mu(s)}{z-s}, \qquad z \in \C \setminus \supp(\mu). 
	\end{equation}

\begin{proof}[Proof of the equalities in \eqref{eq:UmuO_outside}
	and \eqref{eq:UmuO_inside}.]

In view of \eqref{eq:Umu-zderivative} we find from 
\eqref{eq:muOmega_Stieltjes1}  that
$\frac{\partial}{\partial z} U^{\mu_\Omega} = \frac{\partial}{\partial z} U^{\mu_V}$ on $\C \setminus \Omega$. 
Since both $U^{\mu_\Omega}$ and $U^{\mu_V}$ are real-valued, 
\[
\frac{\partial}{\partial\bar{z}}U^{\mu_\Omega}(z) = \overline{\frac{\partial}{\partial z}U^{\mu_\Omega}(z)} \quad\text{ and }\quad \frac{\partial}{\partial\bar{z}}U^{\mu_V}(z) = \overline{\frac{\partial}{\partial z}U^{\mu_V}(z)}.
\] 
and also the $\bar{z}$--derivatives coincide. We  conclude that $U^{\mu_{\Omega}}- U^{\mu_V}$ is constant on $\C \setminus \Omega$. 
As both  $\mu_V$ and $\mu_{\Omega}$ are probability measures, we have 
$\lim\limits_{z\to\infty}(U^{\mu_\Omega}(z)-U^{\mu_V}(z)) = 0$.
Therefore the constant is zero, and 
the equality in \eqref{eq:UmuO_outside} for $z \in \C \setminus \Omega$ follows.

\medskip
Applying $-2 \frac{\partial}{\partial z}$ to the left-hand side of
\eqref{eq:UmuO_inside} we obtain by \eqref{eq:Umu-zderivative}
\[ \int \frac{d\mu_{\Omega}(s)}{z-s}
	+ \frac{2az}{z^2+b^2} - \frac{(1+2a) \bar{z}}{1+|z|^2},
		  \]
which is $0$ for $z \in \Omega$, by \eqref{eq:muOmega_Stieltjes2}. 
Also the $\bar{z}$--derivative vanishes in $\Omega$, 
since the left-hand side of \eqref{eq:UmuO_inside} is real-valued, and 
	the equality in \eqref{eq:UmuO_inside} for $z \in \Omega$,
follows for some integration constant $c$.
\end{proof}

\section{Proof of the inequalities in Proposition \ref{prop:UmuO}}
\label{sec:inequalities}

In this section we prove the inequalities in Proposition \ref{prop:UmuO}.
For the proof we introduce a dynamical picture in which we
introduce a time-like parameter $t$ and we investigate
how the ellipses $\Omega$ and the measures $\mu_\Omega$ and $\mu_V$
evolve with $t$. 

As before we consider $b > 1$ to be fixed and $a > a_{cr} = (b^2-1)^{-1}$
will vary. To $a$ we associate the parameter 
\begin{equation} \label{eq:def_tau} 
		t_a = t = \frac{1}{1+2a}.
	\end{equation}
Since $\mu_{\Omega}$ is a probability measure we see from
\eqref{eq:muOmega} that
\begin{equation} \label{eq:def_tau2}
	t = 	\int_{\Omega} \frac{dA(z)}{\pi (1+|z|^2)^2}.
	\end{equation} 
If $a$ decreases from $\infty$ to $a_{cr}$ 
then $t$ increases from $0$ to
\begin{equation} \label{eq:def_taucr} 
	t_{cr} = \frac{b^2-1}{b^2+1} < 1.
\end{equation}

\subsection{Family of measures $\rho_t$}

For each $t\in (0, t_{cr})$, we use $\Omega(t)$ do
denote the region enclosed by the ellipse \eqref{eq:equation_ellipse}
with $a = \frac{1-t}{2t}$, and \eqref{eq:def_tau}  holds.

As $t$ increases, $a$ decreases and, since $b>1$, the coefficients of $x^2$ 
and $y^2$ in \eqref{eq:equation_ellipse} also decrease. 
Hence the two semi--axes of the ellipse increase as $t$ increases in 
a strictly monotone and continuous way, starting from zero at $t =0$. 
Therefore, we can write $\Omega(t)$ as a disjoint union of ellipses
\begin{equation}\label{eq:decomposition_ellipses}
\Omega(t) = \bigcup_{s\in[0,t]}\partial\Omega(s),
\end{equation} 
where we set $\Omega(0) := \{(0,0)\}$. 

For each $t \leq t_{cr}$, we have the measure $\mu_{\Omega(t)}$ and by \eqref{eq:muOmega} 
\[ t d\mu_{\Omega(t)}(z) =  \frac{dA(z)}{\pi (1+|z|^2)^2}\Big|_{\Omega(t)}.
\]
Thus $\left(t \mu_{\Omega(t)}\right)_t$ is an increasing family of measures. Then
the derivative
\begin{equation} \label{eq:rhot} 
	\rho_t = \frac{\partial}{\partial t} (t \mu_{\Omega(t)}) = 
	\lim_{h \to 0} \frac{(t+h) \mu_{\Omega(t+h)} - t \mu_{\Omega(t)}}{h} 
	\end{equation}
exists for almost every $t \in (0, t_{cr})$ (by general theory as in \cite{BuyarovRakhmanov1999})), where the limit is in the weak$^*$ sense. In our case the limit
in \eqref{eq:rhot} exists for every $t \in (0, t_{cr})$ and defines a probability measure $\rho_t$ that is supported
on $\partial \Omega(t)$.  The co-area formula, 
see e.g.\ \cite[Chapter IV.1, Theorem 1]{Chavel1984},  
provides an explicit expression for $\rho_t$
\begin{equation} \label{eq:rhot_coarea} 
	d\rho_t(z) = \frac{1}{\pi \left| \grad u(z) \right|} \frac{d\nu_t(z)}{(1+|z|^2)^2}, \qquad z \in \partial \Omega(t), \end{equation}
in terms of the arclength $\nu_t$ on $\partial \Omega(t)$, 
and $u : \Omega(t_{cr}) \to \R^+$ is the mapping that assigns $u(z) = t$ to $z \in \partial \Omega(t)$. We will not use
the formula \eqref{eq:rhot_coarea}.

We recover $\mu_{\Omega(t)}$ by integration of \eqref{eq:rhot}
(i.e., the Fundamental Theorem of Integral Calculus) 
\[ t \mu_{\Omega(t)} = \int_0^{t} \rho_s ds. \]
In particular
\begin{equation} \label{eq:UmuOmega_integral} 
	t U^{\mu_{\Omega(t)}}(z) = \int_0^t U^{\rho_s}(z) ds,
		\qquad z \in \C.
\end{equation}

\begin{lemma} \label{eq:rhot_balayage}
	For every $t \in (0,t_{cr})$ we have that $\rho_t$ is the balayage of 
	$\frac{1}{2}(\delta_{ib} + \delta_{-ib})$ onto $\Omega(t)$.
	\end{lemma}
\begin{proof}
	From \eqref{eq:UmuO_inside} with $t = \frac{1}{1+2a}$, we have
	with a constant $c = c(t)$ that will depend on $t$,
	\begin{equation} \label{eq:tUmuOt} 
		t U^{\mu_{\Omega(t)}}(z)
		+ \frac{1-t}{2} \log \frac{1}{|z^2+b^2|}
			+ \frac{1}{2} \log(1+|z|^2) = c(t),
				\qquad z \in \Omega(t). \end{equation}	
	Taking the derivative of \eqref{eq:tUmuOt} 
	with respect to $t$, we obtain
	\begin{equation} \label{eq:Urhot_balayage} 
		U^{\rho_t}(z) = \frac{1}{2} U^{\delta_{ib} + 		\delta_{-ib}}(z) 
		+ c'(t), \qquad  z \in \Omega(t), \end{equation}
	since 
	\[	U^{\delta_{ib} + \delta_{-ib}}(z)
	= \log \frac{1}{|z-ib|} + \log \frac{1}{|z+ib|}
	= \log \frac{1}{|z^2+b^2|}. \]	
	The property \eqref{eq:Urhot_balayage} characterizes
	the  balayage measure,
	and the lemma follows.
\end{proof}

\subsection{Family of measures $\omega_t$}
Next we obtain a similar decomposition of $\mu_V$. 
To indicate the dependence of $V$, see \eqref{eq:def_V}, 
on the parameter $t = \frac{1}{1+2a}$, we write
\begin{equation} \label{eq:Vtau}
	V = V(t) = \frac{1}{4t}\left[(1+t)\log(x^2+b^{-2})+(1-t)\log(x^2+b^2)\right].
\end{equation}
Accordingly, we denote by $\mu_{V(t)}$ the equilibrium measure in the external field
$V(t)$ on the real line, and by $A(t)$ and $C(t)$ the constants 
in \eqref{eq:AC_postcrit}, but depending on $t \in (0, t_{cr})$.

\begin{lemma}\label{lem:deriv_muv} The following hold. 
	\begin{enumerate}
	\item[\rm (a)] $(t \mu_{V(t)})_t$ with $0 < t < t_{cr}$ is an increasing family of measures
	\item[\rm (b)] The derivative
	\begin{equation} \label{eq:omegat} 
		\omega_t = \frac{\partial}{\partial t} (t \mu_{V(t)}) 
		\end{equation}
	exists for every $t \in (0, t_{cr})$ and is a probability measure on $[-A(t), A(t)]$.
	\item[\rm (c)] $\omega_t$ is a balayage measure of point masses onto $[-A(t), A(t)]$,
	\[ \omega_t = \frac{1}{2} \Bal(\delta_{ib} + \delta_{ib^{-1}}, [-A(t),A(t)]) \]
	\end{enumerate}	
\end{lemma}

\begin{proof} The proof is by calculation. 
Let $0 < t < t_{cr}$.	
	By \eqref{eq:density_postcrit} and \eqref{eq:AC_postcrit} we have
	\[ \pi (x^2+b^2)(x^2+b^{-2})
	\frac{d\mu_{V(t)}}{dx} = \sqrt{C(t)} \sqrt{A(t)^2-x^2},
	\quad x \in [-A(t), A(t)], \]
with $C(t) = \frac{b^4-1}{b^4} \frac{b^4(1-t)^2-(1+t)^2}{4t^2}$
and $A(t)^2  = \frac{4b^2t}{b^4(1-t)^2-(1+t)^2}$.
Then for $x \in (-A(t), A(t))$,
\begin{align} 
	\frac{d}{d t} \left(t  \frac{d\mu_{V(t)}}{dx} \right) 
	& = \frac{1}{2 \pi \sqrt{A(t)^2-x^2}}
			\left[\frac{b \sqrt{A(t)^2+b^2}}{x^2+b^2}
				+ \frac{b^{-1} \sqrt{A(t)^2+b^{-2}}}{x^2+ b^{-2}} \right]
			\label{eq:omegat_density}
\end{align}
which is clearly positive, and it gives part (a) of  Lemma
\ref{lem:deriv_muv}. Part (b)
also follows, where $\omega_t$ is the probability measure  with density \eqref{eq:omegat_density} on $[-A(t),A(t)]$.

The balayage of a point mass $\delta_{ic}$, $c > 0$, onto
the real interval $[-A,A]$ has the density
\[  \frac{d}{dx} \Bal(\delta_{ic},[-A,A])
	=  \frac{c \sqrt{A^2+c^2}}{\pi (x^2+c^2)\sqrt{A^2-x^2}}. \]
Comparing this with \eqref{eq:omegat_density} we see that $\omega_t$ 
is indeed  the average of balayage measures as claimed in part (c).
\end{proof}

It follows from part (b) of Lemma \ref{lem:deriv_muv} that
for every $t \in (0, t_{cr})$,
\[ t \mu_{V(t)} = \int_0^{t} \omega_s ds \]
and in particular
\begin{equation} \label{eq:UmuV_integral} 
	t U^{\mu_{V(t)}}(z) = \int_0^{t} U^{\omega_s}(z) ds,
	\qquad z \in \C. 
\end{equation}
which is the analogue of \eqref{eq:UmuOmega_integral}.

\subsection{Proof of inequality in \eqref{eq:UmuO_outside}}

We let $t_a = \frac{1}{1+2a}$ so that $\Omega = \Omega(t_a)$
and $V = V(t_a)$.

Let  $t < t_a$.  Then by  \eqref{eq:UmuO_outside} 
\[ U^{\mu_{\Omega(t)}}(z) = U^{\mu_{V(t)}}(z), \qquad z \in \C \setminus \Omega(t). \]
By \eqref{eq:UmuOmega_integral} and \eqref{eq:UmuV_integral}
we then have 
\[ \int_0^t U^{\rho_s}(z) ds = \int_0^t U^{\omega_s}(z) ds,
	\qquad z \in \C \setminus \Omega(t). \]
Taking the derivative with respect to $t$,
we find
\[ U^{\rho_{t}}(z) = U^{\omega_{t}}(z), \]
for $z \in \C \setminus \Omega(t)$, and by continuity of
the logarithmic potentials also for $z \in \partial \Omega(t)$.
Since $\rho_{t}$ is a probability measure on $\partial \Omega(t)$,
we conclude that $\rho_{t}$ is the balayage of $\omega_{t}$ onto 
$\partial \Omega(t)$.

In the interior of $\Omega(t)$, the logarithmic potential
$U^{\rho_{t}}$ is harmonic, since $\rho_t$ is supported on
$\partial \Omega(t)$, while  $U^{\omega_{t}}$ is superharmonic.
Thus $U^{\omega_{t}} - U^{\rho_{t}}$ is superharmonic in
the interior and zero on the boundary of $\Omega(t)$. 
By the Minimum Principle for superharmonic functions
it then follows that $U^{\rho_t} \leq U^{\omega_t}$ on
$\Omega(t)$, and then in fact on all of $\C$.

Then integrating the inequality from $0$ to $t_a$
and using \eqref{eq:UmuOmega_integral} and \eqref{eq:UmuV_integral}
we obtain the inequality in \eqref{eq:UmuO_outside}.

\subsection{Proof of inequality in  \eqref{eq:UmuO_inside}}

Let again $t_a = \frac{1}{1+2a}$ so that $\Omega = \Omega(t_a)$.
The inequality \eqref{eq:UmuO_inside} comes down to proving 
\begin{equation} \label{eq:UmuO_toprove} 
	t U^{\mu_{\Omega(t_a)}}(z)
		+ \frac{1-t}{2} \log \frac{1}{|z^2+b^2|}
			+ \frac{1}{2} \log(1+|z|^2) \geq c(t),
			\qquad z \in \C,
\end{equation}
for $t=t_a$, where the constant $c(t)$ is such that equality holds on $\Omega(t)$,
see also \eqref{eq:tUmuOt}.

As $t \to t_{cr}$ we recall that $\Omega(t)$ tends to the
horizontal strip \eqref{eq:Omegacr} and  
\[ \mu_{\Omega(t)} \to \mu_{\Omega(t_{cr})}
	= \left. \frac{dA(z)}{\pi (1+|z|^2)^2} \right|_{\Omega_{cr}}. \]
The inequality \eqref{eq:UmuO_toprove} with $t = t_{cr}$
follows from the work of Brauchart et al. \cite{Brauchart2014}

Recall from \eqref{eq:Urhot_balayage} that we have
\[ U^{\rho_t}(z) = \frac{1}{2} \log \frac{1}{|z^2+b^2|}  + c'(t), \qquad z \in \Omega(t). \]
Outside $\Omega(t)$ we have inequality 
\begin{equation} \label{eq:Urhot_inequality} 
	U^{\rho_t}(z) \leq \frac{1}{2} \log \frac{1}{|z^2+b^2|}  + c'(t), \qquad z \in \C\setminus\Omega(t), \end{equation}
as can be seen from the Minimum Principle for superharmonic functions
applied to the function $z \mapsto
\frac{1}{2} \log \frac{1}{|z^2+b^2|} - U^{\rho_t}(z)$
on $\C \setminus \Omega(t)$.

Then we have by \eqref{eq:UmuOmega_integral}
\[ t_a U^{\mu_{\Omega(t_a)}}(z) = 
	t_{cr} U^{\mu_{\Omega(t_{cr})}}(z)
	- \int_{t_a}^{t_{cr}} U^{\rho_t}(z) dt. \]
Using \eqref{eq:Urhot_inequality} and the fact that 
\eqref{eq:UmuO_toprove} holds for $t_{cr}$,   we can estimate
\begin{align*} 
	t_a U^{\mu_{\Omega(t_a)}}(z) & \geq
		- \frac{1-t_{cr}}{2} \log \frac{1}{|z^2+b^2|}
		- \frac{1}{2} \log(1+|z|^2) + c(t_{cr})
		- \int_{t_a}^{t_{cr}} 
	\left( \frac{1}{2} \log \frac{1}{|z^2+b^2|}  + c'(t) \right) dt \\
	 & = - \frac{1-t_a}{2} \log \frac{1}{|z^2+b^2|}
	 	- \frac{1}{2} \log(1+|z|^2) + c(t_a)  
\end{align*}
which is \eqref{eq:UmuO_toprove} with $t=t_a$.

\appendix
\section{Proof of \eqref{eq:musigma} and \eqref{eq:volume_droplet}} \label{subsec:musigma_lemma}

\begin{lemma}
	Suppose $\sigma$ is a measure on $\S^2$ such that
	\eqref{eq:disjoint_supports} hold. Let $D = \supp(\mu_{\sigma})$. 
	Then \eqref{eq:musigma} and \eqref{eq:volume_droplet}  hold. 
\end{lemma}

\begin{proof}
	It is well-known that for any measure $\mu$ on $\S^2$,
	\[ \frac{1}{2\pi} \Delta U^{\mu} = \mu - \m(\mu) \lambda \]
	in the distributional sense, where we use $\Delta = - \div \grad$ for the Riemannian Laplacian on $\S^2$. 
	We apply this to $\sigma$ and $\mu_{\sigma}$ to find
	\begin{equation} \label{eq:DeltaU}
	\frac{1}{2\pi} \Delta U^{\sigma}  = 
	\sigma - \m(\sigma) \lambda, \qquad 
	\frac{1}{2\pi} \Delta U^{\mu_\sigma}  = 
	\mu_\sigma -   \lambda, 
	\end{equation}
	By the variational conditions \eqref{eq:frostmaneasy}
	the sum $U^{\sigma} + U^{\mu_{\sigma}}$ is constant on the interior $\text{int}(D)$,
	and therefore $\Delta(U^{\sigma} + U^{\mu_{\sigma}}) = 0$
	on $\text{int}(D)$. Thus by \eqref{eq:DeltaU}
	\[ \sigma - \m(\sigma) \lambda + \mu_\sigma -   \lambda = 0 \quad \text{on } \text{int}(D). \] 
	Since $\supp(\sigma) \cap D = \emptyset$ this implies that
	$\mu_{\sigma} = (1+\m(\sigma)) \lambda$
	on $\text{int}(D)$. From \cite[Theorem 11]{Gustafsson2018} (or Theorem 6.3 in the arXiv version) it follows that $\lambda(\partial D) = 0$ and hence $\mu_\sigma = (1+\m(\sigma))\lambda$ on $D$. Since $\mu_{\sigma}$ is a probability measure on $D$
	we then have
	\[ \mu_{\sigma} = (1+\m(\sigma)) \lambda_D \]
	and $1 = \m(\mu_{\sigma}) = 
	(1+\m(\sigma)) \lambda(D)$.
	Then \eqref{eq:musigma} and \eqref{eq:volume_droplet} follow.
\end{proof}

\section*{Acknowledgements}

The authors are supported by FWO Flanders project EOS 30889451.
The second author is also supported by long term structural funding-Methusalem grant
of the Flemish Government, and by FWO Flanders project G.0864.16.

\bibliographystyle{amsplain}
\bibliography{refs}
\end{document}